\documentclass[12pt,a4]{amsart}
\usepackage{amsthm,amsmath,amssymb}
\usepackage{titlesec}
\usepackage{mathrsfs}
\makeatletter
\def\section{\@startsection{section}{1}%
  \z@{.7\linespacing\@plus\linespacing}{.5\linespacing}%
  {\normalfont\scshape\centering}}
\def\subsection{\@startsection{subsection}{2}%
  \z@{.5\linespacing\@plus.7\linespacing}{-.5em}%
  {\normalfont\bfseries}}
\makeatother
\titleformat*{\section}{\large\bfseries}
\titleformat*{\subsection}{\large\bfseries}
\newtheorem{theorem}{Theorem}[section]

\newtheorem{lemma}{Lemma}[section]
\newtheorem{prop}{Proposition}[section]

\theoremstyle{remark}
\newtheorem{rem}{Remark}

\title
{On a generalisation of the Riemann $\xi$-function}
\author{Hirotaka Kobayashi}
\date{}
\address{Graduate School of Mathematics, Nagoya University, Furocho, Chikusaku, Nagoya 464-8602, Japan}
\email{m17011z@math.nagoya-u.ac.jp}
\begin{document}

\maketitle

\begin{abstract}
It is known that we can construct the meromorphic function $Z_k(s)$ associated with the higher derivative of Hardy's $Z$-function.
In this paper, we introduce the entire function derived from $Z_k(s)$, a generalisation of the Riemann $\xi$-function and prove some properties.
\end{abstract}

\section{Introduction}
Hardy's $Z$-function is defined by
\begin{equation*}
Z(t)=\chi\left(\frac{1}{2}+it \right)^{-\frac{1}{2}}\zeta\left(\frac{1}{2}+it\right) \quad (t \in \mathbb{R}),
\end{equation*}
where $\chi(s)=2^s\pi^{s-1}\sin(\pi s/2)\Gamma(1-s)$ which comes from the functional equation for $\zeta(s)$.
When we take the higher derivative of $Z(t)$, it is possible to construct a meromorphic function $Z_k(s)$ which satisfies
\begin{equation*}
\left|Z^{(k)}(t) \right|=\left|Z_k\left(\frac{1}{2}+it \right) \right|.
\end{equation*}

This function has some similar properties to those of $\zeta(s)$.
Matsumoto and Tanigawa \cite{M-T} studied the zeros of $Z_k(s)$.
Later, Matsuoka \cite{M} considered the relation between the zero of $Z^{(k)}(t)$ and that of $Z^{(k+1)}(t)$.
They obtained some properties of $Z_k(s)$.
In later sections we will present their results. Especially, we reproduce the proofs of Matsuoka's results because his preprint \cite{M} is unpublished.

In this article, we consider $Z_k(s)$ rather than $Z^{(k)}(t)$. Actually, we generalise the Riemann $\xi$-function to define
\begin{equation*}
\xi_k(s):=\pi^{-\frac{s}{2}}s(s-1)\frac{Z_k(s)}{\Gamma(\frac{s}{2})^{k-1}\Gamma(\frac{1-s}{2})^k},
\end{equation*}
and give some properties of $\xi_k(s)$.
When $k=0$, this coincides with the Riemann $\xi$-function. We may expect that $\xi_k(s)$ is a natural generalisation of the Riemann $\xi$-function because the facts on $\xi_k(s)$ proved in this note are valid for $k=0$.
We will prove that this function is entire in Section 3.
For entire functions, the concept of the order is defined. The order of $\xi_k(s)$ is $1$. Thus we can factorize this function by Hadamard's factorization theorem and we can determine some constants appear in the factorized form.
We show these facts in Section 5.
In Section 6, we give the special values of $\xi_k(s)$ at positive integers. As we see in Section 2, $\xi_k(s)$ has the functional equation. It implies that we can obtain the special values at all integers.
Finally, we show Matsuoka's theorem, or more precisely, Mozer's formula by a different way from Matsuoka's proof. The factorization of $\xi_k(s)$ leads to this different way. This method is a generalisation of the method by which Edwards (\cite{Ed}) and Mozer (\cite{Mo}) (maybe independently) proved the special case of Matsuoka's result.

\begin{rem}
Actually, Matsumoto and Tanigawa considered the function denoted by $\eta_k(s)$ in their notation. However, this function is essentially the same as $Z_k(s)$.
In fact, we can show that $-\frac{1}{2}\omega(s)\eta_k(s)=Z_k(s)$.
\end{rem}

\section{The definition and basic properties of $Z_k(s)$}
Let $Z_0(s)=\zeta(s)$, and according to Matsuoka \cite{M} we define $Z_k(s)$ for $k\geq 1$ by
\begin{equation}\label{def}
Z_{k+1}(s)=Z_k'(s)-\frac{1}{2}\omega(s)Z_k(s) \quad (k\geq 0),
\end{equation}
where
\begin{equation*}
\omega(s):=\frac{\chi'}{\chi}(s)=\log 2\pi-\frac{\Gamma'}{\Gamma}(s)+\frac{\pi}{2}\tan\frac{\pi s}{2}.
\end{equation*}

\begin{rem}
The function $Z_k(s)$ was already introduced in 1990 by Y\i ld\i r\i m \cite{Y} in the following alternative way:
\begin{equation}\label{Y's-def}
Z_k(s):=(\chi(s))^{\frac{1}{2}}\frac{d^k}{ds^k}\left((\chi(s))^{-\frac{1}{2}}\zeta(s) \right).
\end{equation}
We can see that this definition coincides with Matsuoka's.
\end{rem}

We can see that
\begin{equation*}
\omega(s)=\omega(1-s).
\end{equation*}

When we put
\begin{equation*}
\theta(t):=\arg \Gamma \left(\frac{1}{4}+\frac{it}{2} \right)-\frac{t}{2}\log \pi
\end{equation*}

then we see that $\chi(1/2+it)=e^{-2i\theta(t)}$, hence 
\begin{equation}\label{omegatheta}
\omega \left(\frac{1}{2}+it \right)=-2\theta'(t).
\end{equation}

\begin{prop}[Proposition 2.1 in \cite{M}]
For any non-negative $k$, we have
\begin{equation}\label{Z^k}
Z^{(k)}(t)=i^{k}Z_k\left(\frac{1}{2}+it \right)e^{i\theta(t)}.
\end{equation}
\end{prop}

\begin{proof}
The case $k=0$ is the definition of $Z(t)$.
If we assume that the equation is true for $k$, then 
\begin{equation}
Z^{(k+1)}(t)=i^{k+1}e^{i\theta(t)}\left(Z_k'\left(\frac{1}{2}+it\right)+\theta'(t)Z_k\left(\frac{1}{2}+it \right)\right).
\end{equation}
By $(\ref{def})$ and $(\ref{omegatheta})$, we find that the equation is true for $k+1$.
\end{proof}

\begin{prop}[The Functional Equation, Proposition 2.2 in \cite{M}]
For any non-negative $k$, we have
\begin{equation}\label{f-e}
\chi(s)Z_k(1-s)=(-1)^k Z_k(s).
\end{equation}
\end{prop}
\begin{proof}
The case $k=0$ is the functional equation of the Riemann $\zeta$-function.
If we assume that the equation is true for $k$, then by the definition,
\begin{align*}
\chi(s)Z_{k+1}(1-s) &= \chi(s)\left(Z_{k}'(1-s)-\frac{1}{2}\omega(1-s)Z_k(1-s) \right) \\
&= \chi'(s)Z_k(1-s)-(-1)^{k}Z_k'(s)-\frac{(-1)^k}{2}\omega(s)Z_k(s) \\
&= (-1)^{k+1}Z_k'(s)+(-1)^k\omega(s)Z_k(s)-\frac{(-1)^k}{2}\omega(s)Z_k(s) \\
&= (-1)^{k+1}\left(Z_k'(s)-\frac{1}{2}\omega(s)Z_k(s) \right) \\
&= (-1)^{k+1}Z_{k+1}(s).
\end{align*}
The proof is completed.
\end{proof}

Therefore we can obtain the following theorem.
\begin{theorem}
For all $s$, we have
\begin{equation}\label{fcte-xi}
\xi_k(s)=(-1)^k\xi_k(1-s).
\end{equation}
\end{theorem}
\begin{proof}
We can transform $\chi(s)$ to
\begin{equation*}
\chi(s)=\frac{\Gamma (\frac{1-s}{2})}{\Gamma(\frac{s}{2})}\pi^{s-\frac{1}{2}}.
\end{equation*}
Thus
\begin{align*}
\xi_k(s)
&=
\pi^{-\frac{s}{2}}s(s-1)\frac{Z_k(s)}{\Gamma(\frac{s}{2})^{k-1}\Gamma(\frac{1-s}{2})^k} \\
&=
\pi^{-\frac{s}{2}}s(s-1)\frac{(-1)^k\chi(s)Z_k(1-s)}{\Gamma(\frac{s}{2})^{k-1}\Gamma(\frac{1-s}{2})^k} \\
&=
(-1)^k\pi^{-\frac{1-s}{2}}s(s-1)\frac{Z_k(1-s)}{\Gamma(\frac{1-s}{2})^{k-1}\Gamma(\frac{s}{2})^k} \\
&=
(-1)^k\xi_k(1-s).
\end{align*}
\end{proof}

Let $f_0(s)=1$, and define $f_k(s)$ for $k\geq 1$ by
\begin{equation}
f_{k+1}(s)=f_k'(s)-\frac{1}{2}\omega(s)f_k(s) \quad (k\geq 1).
\end{equation}

The following proposition is inspired by  Proposition 3 in \cite{M-T}.
\begin{prop}\label{exZ_k}
For any non-negative $k$, we have
\begin{equation}
Z_k(s)=\sum_{j=0}^{k}\binom{k}{j}f_{k-j}(s)\zeta^{(j)}(s).
\end{equation}
\end{prop}

\begin{proof}
The case $k=0$ is clear. We assume that this is valid for $k$. By the definition,
\begin{align*}
Z_{k+1}(s)&=Z_k'(s)-\frac{1}{2}\omega(s)Z_k(s)\\
&=\sum_{j=0}^{k}\binom{k}{j}f_{k-j}'(s)\zeta^{(j)}(s)+\sum_{j=0}^{k}\binom{k}{j}f_{k-j}(s)\zeta^{(j+1)}(s) \\
&\quad-\frac{1}{2}\omega(s)\sum_{j=0}^{k}\binom{k}{j}f_{k-j}(s)\zeta^{(j)}(s)\\
&=\sum_{j=0}^{k}\binom{k}{j}f_{k+1-j}(s)\zeta^{(j)}(s)+\sum_{j=0}^{k}\binom{k}{j}f_{k-j}(s)\zeta^{(j+1)}(s)\\
&=\sum_{j=0}^{k+1}\left\{\binom{k}{j}+\binom{k}{j-1}\right\}f_{k+1-j}(s)\zeta^{(j)}(s)\\
&=\sum_{j=0}^{k+1}\binom{k+1}{j}f_{k-j}(s)\zeta^{(j)}(s).
\end{align*}
Here to obtain the last equality, we use the relation
\begin{equation*}
\binom{k}{j}+\binom{k}{j-1}=\binom{k+1}{j}.
\end{equation*}
\end{proof}

The function $f_k(s)$ can be expressed explicitly as follows.

\begin{prop}\label{exf_k}
For $k\geq 1$, we have
\begin{equation*}
f_k(s)=k!\sum_{a_1+2a_2+\dots+ka_k=k}\left(-\frac{1}{2} \right)^{a_1+\dots+a_k}\prod_{l=1}^{k}\frac{1}{a_{l}!}\left(\frac{\omega^{(l-1)}(s)}{l!} \right)^{a_l}.
\end{equation*}
\end{prop}
We can guess this expression by (\ref{Y's-def}) and Fa\`{a} di Bruno's formula (see \cite{J}, \cite{R}).

\begin{proof}
We will prove this statement by induction on $k$. When $k=1$, this formula is trivial.
We assume that this is true for $k=j$. By the definition, we have
\begin{align*}
f_{j+1}(s)
&=
f_j'(s)-\frac{1}{2}\omega(s)f_j(s) \\
&=
j!\sum_{a_1+2a_2+\dots+ja_j=j}\left(-\frac{1}{2} \right)^{a_1+\dots+a_j}\sum_{\substack{m=1 \\ a_m\geq 1}}^{j}\prod_{\substack{l=1 \\ l\neq m}}^{k}\frac{1}{a_{l}!}\left(\frac{\omega^{(l-1)}(s)}{l!} \right)^{a_l} \\
&\quad
\times \frac{1}{(a_m-1)!}\left(\frac{1}{m!} \right)^{a_m}\omega^{(m-1)}(s)^{a_m-1}\omega^{(m)}(s) \\
&\quad
-\frac{1}{2}\omega(s)j!\sum_{a_1+2a_2+\dots+ja_j=j}\left(-\frac{1}{2} \right)^{a_1+\dots+a_j}\prod_{l=1}^{j}\frac{1}{a_{l}!}\left(\frac{\omega^{(l-1)}(s)}{l!} \right)^{a_l}.
\end{align*}
On the first term, let $b_m=a_m-1, b_{m+1}=a_{m+1}+1$ and $b_l=a_l \ (l\neq m, m+1)$ for each $m \ (1\leq m\leq j-1)$. When $m=j$, we put $b_j=a_j-1$ and $b_{j+1}=1$. Then we can see that the first term is
\begin{align*}
&\quad
j!\sum_{b_1+2b_2+\dots+jb_j=j+1}\left(-\frac{1}{2} \right)^{b_1+\dots+b_j}\sum_{m=1}^{j-1}(m+1)b_{m+1}\prod_{l=1}^{j}\frac{1}{b_{l}!}\left(\frac{\omega^{(l-1)}(s)}{l!} \right)^{b_l} \\
&\quad
+(j+1)!\sum_{\substack{b_1+2b_2+\dots+(j+1)b_{j+1}=j+1 \\ b_{j+1}=1}}\left(-\frac{1}{2} \right)^{b_1+\dots+b_{j+1}}\prod_{l=1}^{j+1}\frac{1}{b_{l}!}\left(\frac{\omega^{(l-1)}(s)}{l!} \right)^{b_l} \\
&=
(j+1)!\sum_{b_1+2b_2+\dots+(j+1)b_{j+1}=j+1}\left(-\frac{1}{2} \right)^{b_1+\dots+b_{j+1}}\prod_{l=1}^{j+1}\frac{1}{b_{l}!}\left(\frac{\omega^{(l-1)}(s)}{l!} \right)^{b_l} \\
&\quad
-j!\sum_{b_1+2b_2+\dots+(j+1)b_{j+1}=j+1}\left(-\frac{1}{2} \right)^{b_1+\dots+b_{j+1}}b_1\prod_{l=1}^{j+1}\frac{1}{b_{l}!}\left(\frac{\omega^{(l-1)}(s)}{l!} \right)^{b_l}.
\end{align*}
As for the second term, when we put $c_1=a_1+1$ and $c_l=a_l \quad (2\leq l\leq j)$ then
\begin{align*}
j!\sum_{c_1+2c_2+\dots+jc_j=j+1}\left(-\frac{1}{2} \right)^{c_1+\dots+c_j}c_1\prod_{l=1}^{j}\frac{1}{c_{l}!}\left(\frac{\omega^{(l-1)}(s)}{l!} \right)^{c_l}.
\end{align*}
Therefore
\begin{align*}
&\quad
f_{j+1}(s) \\
&=
(j+1)!\sum_{a_1+2a_2+\dots+(j+1)a_{j+1}=j+1}\left(-\frac{1}{2} \right)^{a_1+\dots+a_{j+1}}\prod_{l=1}^{j+1}\frac{1}{a_{l}!}\left(\frac{\omega^{(l-1)}(s)}{l!} \right)^{a_l}.
\end{align*}
This completes the proof.
\end{proof}

\section{The poles of $f_k(s)$ and $Z_k(s)$}
We investigate the poles of $f_k(s)$ and $Z_k(s)$.

\begin{lemma}[Lemma 1 in \cite{M-T}]\label{pole-omega}
The poles of $\omega(s)$ are all simple, and located at positive odd integers with residue $-1$ and at non-positive even integers with residue $1$.
\end{lemma}

\begin{proof}
Since $\chi(s)=2^s\pi^{s-1}\sin (\pi s/2)\Gamma(1-s)$, the zeros of $\chi(s)$ are non-positive even integers and the poles are positive odd integers. Therefore we obtain the lemma.
\end{proof}

\begin{lemma}[Lemma 4.2 in \cite{M}]\label{f_poles}
For $k\geq 0$, the function $f_k(s)$ has poles of order $k$ which are located only at $\dots, -4,-2,0,1,3,5, \dots$.
\end{lemma}

\begin{proof}
The case $k=1$ is the previous lemma.
We assume that the lemma is valid for $k\geq 1$.
Let $a$ be a pole of $f_k(s)$. Then by the Laurent expansion at centre $a$,
we have
\begin{equation*}
f_k(s)=\frac{c_k}{(s-a)^k}+\cdots,
\end{equation*}
where $c_k$ does not vanish.
By the definition and the previous lemma, we have
\begin{equation*}
f_{k+1}(s)=\frac{-kc_k\pm\frac{c_k}{2}}{(s-a)^{k+1}}+\cdots.
\end{equation*}
Since $-kc_k\pm c_k/2\neq 0$, the lemma is true for $k+1$. This completes the proof of the lemma.
\end{proof}

	This lemma and Proposition \ref{exZ_k} immediately lead to the following lemma.
\begin{lemma}[Lemma 4.4 in \cite{M}]\label{poles}
For $k\geq 0$, the function $Z_k(s)$ has poles of order $k$ located at $0,3,5,7, \cdots$,
that of order $k+1$ located at $1$, and those of order $k-1$ located at $-2,-4,-6, \cdots$.
\end{lemma}
We understand that ``poles of order $-1$'' means zeros of order $1$.

This lemma implies $\xi_k(s)$ is entire, because $\Gamma(s/2)$ and $\Gamma((1-s)/2)$ have poles at non-positive even integers and positive odd integers respectively.

\section{Some estimates on $Z_k(s)$}
In this section, we prove some estimates on $Z_k(s)$.
Proposition \ref{exZ_k} implies that we need the estimates of $f_k(s)$ and $\zeta^{(k)}(s)$ for any non-negative integer $k$.

First we consider $f_k(s)$.
Stirling's formula implies that
\begin{lemma}\label{estgam}
For $\sigma \geq \frac{1}{4}$, we have

\begin{equation*}
\Gamma(s)=\sqrt{2\pi}s^{s-\frac{1}{2}}e^{-s}(1+O(|s|^{-1})),
\end{equation*}

\begin{equation*}
\frac{\Gamma'}{\Gamma}(s)=\log s+O(|s|^{-1}) \label{degamma}
\end{equation*}
and
\begin{equation*}
\frac{d^k}{ds^k}\frac{\Gamma'}{\Gamma}(s)=O_k(|s|^{-k}) \quad (k\geq 1).
\end{equation*}
\end{lemma}

Define the set $\mathscr{D}$ by removing all small circles whose centres are odd positive integers and even non-positive integers with radii depending on $k$ from the complex plane.

\begin{lemma}[Lemma 3.1 in \cite{M}]\label{esttan}
In the region $\{s \in \mathscr{D}|\sigma>1/4 \}$, we have
\begin{equation}\label{tan}
\tan \frac{\pi s}{2}=\begin{cases}
           i+O(e^{-2t}) & \text{$(t\geq 0)$}, \\
           -i+O(e^{2t}) & \text{$(t\leq 0)$},
           \end{cases}
\end{equation}
and
\begin{equation}\label{hightan}
\frac{d^k}{ds^k}\tan \frac{\pi s}{2}=O_k(e^{-2|t|}) \quad(k\geq 1).
\end{equation}
\end{lemma}

By this lemma and Lemma $\ref{estgam}$, we have
\begin{lemma}[Lemma 3.2 in \cite{M}]\label{estome}
For $s \in \mathscr{D}$, we have
\begin{equation}\label{omega}
\omega(s)=-\log |s|+O(1),
\end{equation}
and
\begin{equation}\label{highomega}
\omega^{(k)}(s)=O_k(1) \quad(k\geq 1).
\end{equation}
\end{lemma}

Therefore we can see that $f_k(s)\ll_k (\log (|s|+2))^k$ for $s \in \mathscr{D}$ by Proposition \ref{exf_k}.

By Lemma \ref{f_poles},
\begin{equation*}
\left(\cos \frac{\pi s}{2}\right)^kf_k(s)
\end{equation*}
has no pole in the half plane $\Re s\geq \frac{1}{2}$.

From the above, for $\Re s\geq \frac{1}{2}$ there is a constant $C_1=C_1(k)$ such that
\begin{equation}\label{f-est}
\left(\cos \frac{\pi s}{2}\right)^kf_k(s)\ll_k e^{C_1|s|}.
\end{equation}

Next we consider $\zeta^{(k)}(s)$.
Let $k=0$.
It is known that
\begin{equation*}
\zeta(s)=\frac{1}{s-1}-s\int_{1}^{\infty}\frac{x^2-[x]-\frac{1}{2}}{x^{s+1}}dx+\frac{1}{2}.
\end{equation*}
Therefore we can see that when $s$ is not close to $1$,
\begin{equation*}
\zeta(s)\ll_{\eta} |s| \quad (\Re s >\eta)
\end{equation*}
for any $\eta>0$.
Therefore, by Cauchy's integral formula applied to a circle which has centre $s$ and radius $1/\log (|s|+2)$,

\begin{equation*}
\zeta^{(k)}(s)\ll_{\eta'} |s|\log^k (|s|+2) \quad (\Re s>\eta')
\end{equation*}
for $k\geq 1$ and any $\eta'=\eta'(k)>0$ unless $s$ is close to $1$.
Thus for all $\Re s>\eta'$, there is a constant $C_2=C_2(k)$ such that

\begin{equation}\label{zeta-est}
\left(\cos \frac{\pi s}{2} \right)^k(s-1)\zeta^{(k)}(s)\ll_k e^{C_2|s|}.
\end{equation}

\section{The factorization of $\xi_k(s)$}
By the estimates showed in the previous section, we can prove the following theorem.
\begin{theorem}\label{xi-facto}
For $k\geq 0$, there are constants $a_k$ and $B_k$ such that

\begin{equation}\label{factral-xi}
\xi_k(s)=e^{A_k+B_ks}\prod_{\rho_k}\left(1-\frac{s}{\rho_k} \right)e^{\frac{s}{\rho_k}}
\end{equation}

for all $s$. Here the product is extended over all zeros $\rho_k$ of $\xi_k(s)$.

\end{theorem}

We note that it is possible that some of the $\rho_k$'s are real. Actually, Anderson (\cite{An}) proved the existence of real zeros for $k=1$. This fact is reproduced by Hall (\cite{H}).

\begin{proof}
It is sufficient to prove that the order of $\xi_k(s)$ is $1$, for then Hadamard's factorization theorem can be applied.

Let $\Re s \geq \frac{1}{2}$.
By the formulas
\begin{equation*}
\Gamma(s)\Gamma(1-s)=\frac{\pi}{\sin \pi s} \quad \text{and} \quad \Gamma(s)\Gamma \left(s+\frac{1}{2} \right)=\sqrt{\pi}2^{1-2s}\Gamma(2s),
\end{equation*}
we see that
\begin{equation*}
\Gamma \left(\frac{s}{2}\right)\Gamma \left(\frac{1-s}{2} \right)=2^{s-1}\sqrt{\pi}\frac{\Gamma(\frac{s}{2})^2}{\Gamma(s)\cos \frac{\pi s}{2}}.
\end{equation*}

Then, by the definition of $\xi_k(s)$ and Proposition \ref{exZ_k},
\begin{align*}
\xi_k(s)
&=
2^{-k(s-1)}\pi^{-\frac{s+k}{2}}s(s-1)\Gamma \left(\frac{s}{2} \right)\left(\frac{\Gamma(s)}{\Gamma \left(\frac{s}{2} \right)^2} \right)^k\left(\cos \frac{\pi s}{2}\right)^k \\
&\quad
\times\sum_{j=0}^{k}\binom{k}{j}f_{k-j}(s)\zeta^{(j)}(s).
\end{align*}

Lemma \ref{estgam} implies that 
\begin{equation}\label{xi_app}
\begin{split}
\xi_k(s)
&=
2^{-\frac{s+k}{2}+1}\pi^{-\frac{s-1}{2}-k}s^{\frac{s+k+1}{2}}e^{-\frac{s}{2}} \\
&\quad
\times\sum_{j=0}^{k}\binom{k}{j}\left(\cos \frac{\pi s}{2} \right)^kf_{k-j}(s)(s-1)\zeta^{(j)}(s)(1+O_k(|s|^{-1})).
\end{split}
\end{equation}
Finally, by the estimates in the previous section (\ref{zeta-est}) and (\ref{f-est}), we have
\begin{equation*}
\xi_k(s) \ll_k \exp(C|s|\log (|s|+2)),
\end{equation*}
where the constant $C$ depends only on $k$. The functional equation for $\xi_k(s)$ (\ref{fcte-xi}) implies that this estimate is valid for $\Re s\leq \frac{1}{2}$ too.
Thus the order of $\xi_k(s)$ is at most $1$.
By (\ref{xi_app}), we have $\log \xi_k(\sigma)\sim \frac{1}{2}\sigma\log \sigma$ for $\sigma\rightarrow \infty$.
This completes the proof.
\end{proof}

We can determine constants $e^{A_k}$ and $B_k$ in (\ref{factral-xi}).
\begin{theorem}
In the previous theorem, for $k\geq 0$ we have
\begin{equation*}
e^{A_k}=\xi_k(0)=\frac{(-1)^k(2k-1)!!}{(4\sqrt{\pi})^k}
\end{equation*}
and
\begin{equation*}
B_k=-\frac{2k(k-1)}{2k-1}\log 2-\frac{1}{2(2k-1)}\log 4\pi+\frac{\gamma}{2(2k-1)}-1,
\end{equation*}
where $\gamma$ is the Euler constant.
\end{theorem}

Here, we note that
\begin{equation*}
(2k-1)!!=\prod_{l=1}^k(2l-1)=(2k-1)(2k-3)\cdots3\cdot2\cdot1,
\end{equation*}
and we define $(-1)!!=1$ and $(-3)!!=-1$.

\begin{proof}
The logarithmic derivative leads to
\begin{equation*}
\frac{\xi_k'}{\xi_k}(0)=B_k,
\end{equation*}
by the previous theorem. Thus, by the definition of $\xi_k(s)$,
\begin{equation*}
B_k=
-\frac{1}{2}\log \pi-1+\frac{k}{2}\frac{\Gamma'}{\Gamma}\left(\frac{1}{2} \right)+\lim_{s\rightarrow 0}\left(\frac{1}{s}-\frac{k-1}{2}\frac{\Gamma'}{\Gamma}\left(\frac{s}{2} \right)+\frac{Z_k'}{Z_k}(s)\right).
\end{equation*}

Since
\begin{equation*}
\frac{1}{s}+\frac{\Gamma'}{\Gamma}(s)=\frac{\Gamma'}{\Gamma}(s+1),
\end{equation*}
we have
\begin{align*}
\frac{1}{s}-\frac{k-1}{2}\frac{\Gamma'}{\Gamma}\left(\frac{s}{2} \right)+\frac{Z_k'}{Z_k}(s)
&=
\frac{k}{s}+\frac{Z_k'}{Z_k}(s)-\frac{k-1}{2}\left(\frac{2}{s}+\frac{\Gamma'}{\Gamma}\left(\frac{s}{2} \right)\right) \\
&=
\frac{k}{s}+\frac{Z_k'}{Z_k}(s)-\frac{k-1}{2}\frac{\Gamma'}{\Gamma}\left(\frac{s}{2}+1 \right).
\end{align*}

We consider the logarithmic derivative of $Z_k(s)$.
We recall $Z_k(s)$ has a pole with order $k$ at $s=0$.
Then we can express
\begin{equation*}
Z_k(s)=\sum_{l=-k}^{\infty}a_{k,l}s^l.
\end{equation*}

Then we obtain
\begin{align*}
\frac{k}{s}+\frac{Z_k'}{Z_k}(s)
&=
\frac{k}{s}+\frac{-ka_{k,-k}s^{-k-1}-(k-1)a_{k,-k+1}s^{-k}+O(|s|^{-k+1})}{a_{k,-k}s^{-k}+a_{k,-k+1}s^{-k+1}+O(|s|^{-k+2})} \\
&=
\frac{ka_{k,-k}s^{-k-1}+ka_{k,-k+1}s^{-k}+O(|s|^{-k+1})}{a_{k,-k}s^{-k}+a_{k,-k+1}s^{-k+1}+O(|s|^{-k+2})} \\
&\quad
+\frac{-ka_{k,-k}s^{-k-1}-(k-1)a_{k,-k+1}s^{-k}+O(|s|^{-k+1})}{a_{k,-k}s^{-k}+a_{k,-k+1}s^{-k+1}+O(|s|^{-k+2})} \\
&=
\frac{a_{k,-k+1}+O(|s|)}{a_{k,-k}+a_{k,-k+1}s+O(|s|^{2})}.
\end{align*}

Therefore
\begin{equation*}
\lim_{s\rightarrow 0}\left(\frac{k}{s}+\frac{Z_k'}{Z_k}(s)\right)=\frac{a_{k,-k+1}}{a_{k,-k}}.
\end{equation*}
By the definition (\ref{def}) and Lemma \ref{pole-omega}, we can see that for $k\geq 1$
\begin{equation*}
a_{k,-k}=-(k-1)a_{k-1,-k+1}-\frac{1}{2}a_{k-1,-k+1}=\left(-k+\frac{1}{2} \right)a_{k-1,-k+1},
\end{equation*}
(see the argument in the proof of Lemma \ref{f_poles}) and
\begin{align*}
\frac{a_{k,-k}}{a_{0,0}}
&=
\frac{a_{k,-k}}{a_{k-1,-k+1}}\cdot\frac{a_{k-1,-k+1}}{a_{k-2,-k+2}}\cdots\frac{a_{1,-1}}{a_{0,0}} \\
&=
\left(-k+\frac{1}{2} \right)\left(-k+\frac{3}{2} \right)\cdots \left(-\frac{1}{2}\right)=\left(-\frac{1}{2}\right)^{k}(2k-1)!!.
\end{align*}
Since $a_{0,0}=\zeta(0)=-1/2$, we have
\begin{equation*}
a_{k,-k}=\left(-\frac{1}{2}\right)^{k+1}(2k-1)!!.
\end{equation*}

On the other hand, by the definition (\ref{def}) and Lemma \ref{pole-omega}, we can obtain
\begin{equation*}
a_{k,-k+1}=-\frac{1}{2}(2k-3)a_{k-1,-k+2}-\frac{C_0}{2}a_{k-1,-k+1},
\end{equation*}
where $C_0$ is the constant term of the Laurent expansion of $\omega(s)$ around $0$ and it is $\log 2\pi+\gamma$.
When we put
\begin{equation*}
a'_k=\frac{a_{k,-k+1}}{-\frac{C_0}{2}a_{k-1,-k+1}}
\end{equation*}
we have
\begin{equation*}
a'_k=a'_{k-1}+1=a'_1+k-1.
\end{equation*}
Thus we can see that
\begin{equation*}
a_{k,-k+1}=\left(\frac{\gamma}{\log 2\pi+\gamma}+k-1 \right)(\log 2\pi+\gamma)\left(-\frac{1}{2}\right)^{k+1}\frac{(2k-1)!!}{2k-1}
\end{equation*}
by the fact that $a_{1,0}=(\zeta'(0)- \zeta(0)\log 2\pi- \zeta(0)\gamma)/2=\gamma/4$.

Hence we have
\begin{align*}
B_k
&=
-\frac{1}{2}\log \pi-1+\frac{k}{2}\frac{\Gamma'}{\Gamma}\left(\frac{1}{2} \right)-\frac{k-1}{2}\frac{\Gamma'}{\Gamma}(1) \\
&\quad
+\left(\frac{\gamma}{\log 2\pi+\gamma}+k-1 \right)\frac{\log 2\pi+\gamma}{2k-1} \\
&=
-\frac{2k(k-1)}{2k-1}\log 2-\frac{1}{2(2k-1)}\log 4\pi+\frac{\gamma}{2(2k-1)}-1.
\end{align*}
Here, to show the last equality, we use
\begin{equation*}
\frac{\Gamma'}{\Gamma}(1)=-\gamma \quad \text{and} \quad \frac{\Gamma'}{\Gamma}\left(\frac{1}{2} \right)=-\gamma-2\log 2.
\end{equation*}

Lastly, we calculate $\xi_k(0)=e^{A_k}$.
It is known that
\begin{equation*}
\Gamma(s+1)=s\Gamma(s).
\end{equation*}
Hence we obtain
\begin{align*}
\xi_k(0)
&=
-\Gamma \left(\frac{1}{2} \right)^{-k}\lim_{s\rightarrow 0}\frac{sZ_k(s)}{\Gamma \left(\frac{s}{2} \right)^{k-1}} \\
&=
-\Gamma \left(\frac{1}{2} \right)^{-k}\lim_{s\rightarrow 0}\left(\frac{s}{2}\right)^{k-1}\frac{sZ_k(s)}{\Gamma \left(\frac{s}{2}+1 \right)^{k-1}} \\
&=
-\frac{2^{-k+1}}{\Gamma(\frac{1}{2})^k}\lim_{s\rightarrow 0}s^kZ_k(s)=-2^{-k+1}\pi^{-\frac{k}{2}}a_{k,-k}=\frac{(-1)^k(2k-1)!!}{(4\sqrt{\pi})^k}.
\end{align*}

\end{proof}

\section{The special values of $\xi_k(s)$ at integer points}
In this section, we give an explicit expression of special values of $\xi_k(s)$ at positive integer points.
\begin{theorem}\label{sp-value}
Let $k\geq 0$. We have
\begin{equation*}
\xi_k(1)=\frac{(2k-1)!!}{(4\sqrt{\pi})^k},
\end{equation*}
and for $n\geq 1$,
\begin{align*}
&\quad
\xi_k(2n+1) \\
&=
(-1)^{kn+1}\pi^{-\frac{2n+k}{2}}\frac{(2n+1)!2n}{4^n\cdot n!}\left(\frac{4^{n-1}(n!)^2}{(2n)!}\right)^k(2k-3)!!\zeta(2n+1)
\end{align*}
and
\begin{align*}
\xi_k(2n)
&=
(-1)^{kn}\pi^{-\frac{2n+k}{2}}2n(2n-1)(n-1)!\left(\frac{(2n)!}{4^nn!(n-1)!}\right)^kZ_k(2n).
\end{align*}
\end{theorem}

Before starting the proof, we note that Proposition \ref{exZ_k} and \ref{exf_k} imply
\begin{align*}
Z_k(2n)
&=
\sum_{j=0}^{k}\frac{k!}{j!}\zeta^{(j)}(2n)\sum_{a_1+2a_2+\dots+(k-j)a_{k-j}=k-j}\left(-\frac{1}{2} \right)^{a_1+\dots+a_{k-j}} \\
&\quad
\times \prod_{l=1}^{k-j}\frac{1}{a_{l}!}\left(\frac{\omega^{(l-1)}(2n)}{l!} \right)^{a_l}.
\end{align*}
On the $\omega^{(l-1)}(2n)$, since
\begin{equation*}
\left.\frac{d^{l-1}}{ds^{l-1}}\tan \frac{\pi s}{2}\right|_{s=2n}
=
\begin{cases}
\displaystyle
\frac{(-1)^{m-1}B_{2m}(4^m-1)\pi^{2m-1}}{2m} & l=2m, \\
0 & l=2m+1,
\end{cases}
\end{equation*}
and for $l\geq 2$
\begin{equation*}
\left.\frac{d^{l-1}}{ds^{l-1}}\frac{\Gamma'}{\Gamma}(s)\right|_{s=2n}=(-1)^l(l-1)!\left(\zeta(l)-\sum_{j=1}^{2n-1}\frac{1}{j^l} \right),
\end{equation*}
we have
\begin{equation*}
\omega^{(l-1)}(2n)
=
\begin{cases}
\displaystyle
(-1)^m\frac{B_{2m}(2-4^m)\pi^{2m}}{4m}+(2m-1)!\sum_{j=1}^{2n-1}\frac{1}{j^{2m}} & l=2m, \\
\displaystyle
-(2m)!\left(\sum_{j=1}^{2n-1}\frac{1}{j^{2m+1}}-\zeta(2m+1) \right) & l=2m+1,
\end{cases}
\end{equation*}
where $B_{2m}$ is the $2m$-th Bernoulli number and therefore $\omega^{(l-1)}(2n)>0$.
When $l=1$, we can obtain
\begin{equation*}
\omega(2n)=\log 2\pi+\gamma-\sum_{j=1}^{2n-1}\frac{1}{j}
\end{equation*}
because
\begin{equation*}
\frac{\Gamma'}{\Gamma}(n)=-\gamma+\sum_{j=1}^{n-1}\frac{1}{j}.
\end{equation*}

\begin{proof}[Proof of Theorem \ref{sp-value}]
The case of $s=1$ is trivial.

Let $n\geq 1$. First we consider the case $s=2n$. When $s=2n$ we have
\begin{align*}
\xi_k(2n)
&=
\pi^{-n}2n(2n-1)\frac{Z_k(2n)}{\Gamma(n)^{k-1}\Gamma(\frac{1}{2}-n)^k} \\
&=
(-1)^{kn}\pi^{-\frac{2n+k}{2}}2n(2n-1)(n-1)!\left(\frac{(2n)!}{4^nn!(n-1)!}\right)^kZ_k(2n).
\end{align*}

Next we treat the case $s=2n+1$. In this case, we see that
\begin{align*}
\xi_k(2n+1)
&=
\pi^{-\frac{2n+1}{2}}(2n+1)2n\frac{1}{\Gamma(\frac{2n+1}{2})^{k-1}}\lim_{s\rightarrow 2n+1}\frac{Z_k(s)}{\Gamma(\frac{1-s}{2})^k} \\
&=
\pi^{-\frac{2n+1}{2}}(2n+1)2n\frac{(4^nn!)^{k-1}}{((2n)!)^{k-1}\pi^{\frac{k-1}{2}}} \\
&\quad
\times \lim_{s\rightarrow 2n+1}\frac{\prod_{m=0}^{n-1}(2m+1-s)^k}{2^{kn+k}\Gamma(\frac{2n+3-s}{2})^k}(2n+1-s)^kZ_k(s) \\
&=
(-1)^{kn+k}\pi^{-\frac{2n+k}{2}}(2n+1)2n\frac{(n!)^k(4^nn!)^{k-1}}{2^k((2n)!)^{k-1}} \\
&\quad
\times \lim_{s\rightarrow 2n+1}(s-2n-1)^kZ_k(s).
\end{align*}

We consider the Laurent expansion at $s=2n+1$ i.e.
\begin{equation*}
Z_k(s)=\sum_{l=-k}^{\infty}b_{k,l}(s-2n-1)^l.
\end{equation*}
By the definition (\ref{def}) and Lemma \ref{pole-omega}, we can see that
\begin{equation*}
b_{k,-k}=-(k-1)b_{k-1,-k+1}+\frac{1}{2}b_{k-1,-k+1}=-\left(k-\frac{3}{2} \right)b_{k-1,-k+1}.
\end{equation*}
Thus we have
\begin{align*}
\frac{b_{k,-k}}{b_{0,0}}
&=
\frac{b_{k,-k}}{b_{k-1,-k+1}}\cdot \cdots \cdot \frac{b_{1,-1}}{b_{0,0}} \\
&=
\left(-k+\frac{3}{2} \right)\cdots \left(-\frac{1}{2} \right)\frac{1}{2}=-\left(-\frac{1}{2} \right)^{k}(2k-3)!!.
\end{align*}
Since $b_{0,0}=\zeta(2n+1)$, we obtain
\begin{equation*}
b_{k,-k}=-\left(-\frac{1}{2}\right)^k(2k-3)!!\zeta(2n+1).
\end{equation*}
This leads to
\begin{align*}
&\quad
\xi_k(2n+1) \\
&=
(-1)^{kn+1}\pi^{-\frac{2n+k}{2}}\frac{(2n+1)!2n}{4^n\cdot n!}\left(\frac{4^{n-1}(n!)^2}{(2n)!}\right)^k(2k-3)!!\zeta(2n+1).
\end{align*}

\end{proof}

\section{Necessary Additional statements}
To show Matsuoka's theorem (Theorem \ref{M-th}), we need some additional statements.

By Proposition \ref{exf_k} and Lemma \ref{estome},
\begin{equation}\label{f_kest}
f_k(s)=\left(-\frac{\omega(s)}{2}\right)^k(1+O_k((\log |s|)^{-2}))
\end{equation}
for $s\in \mathscr{D}(2m)=\{ s\in \mathscr{D} \mid \sigma \geq 2m \}$ with sufficiently large integer $m=m(k)$.
This implies $\Re f_k(s)>0$ in this region.
Thus, by the argument principle, we have

\begin{lemma}[Lemma 4.3 in \cite{M}]\label{f_zeros-outside}
For $s\in \{ s \mid \sigma<1-2m \} \cup \{ s \mid \sigma>2m \}$, the zeros of $f_k(s)$ are in small circles centred negative even integers and positive odd integers,
and the number of those is $k$ in each circles.
\end{lemma}

By Proposition \ref{exZ_k}, (\ref{f_kest}) and the facts that
\begin{equation*}
\zeta(s)=1+O(2^{-\sigma}), \quad \zeta^{(k)}=O_k(2^{-\sigma}),
\end{equation*}
we see that
\begin{align*}
Z_k(s)
&=
f_k(s)\left \{\zeta(s)+\sum_{j=1}^{k}\binom{k}{j}\frac{f_{k-j}(s)}{f_k(s)}\zeta^{(j)}(s) \right \} \\
&=
f_k(s)\{1+O(2^{-\sigma})\}
\end{align*}
for $s\in \mathscr{D}(2m)=\{ s\in \mathscr{D} \mid \sigma \geq 2m \}$ and $k\geq 1$.

Hence, by (\ref{f_kest})
\begin{equation}\label{Zkest}
Z_k(s)=\left(-\frac{\omega(s)}{2}\right)^k\{1+O_k((\log |s|)^{-2}) \} \quad(k\geq 1)
\end{equation}
for $s\in \mathscr{D}(2m)$. This leads to $\Re Z_k(s)>0$ in this region.
Hence, by the same argument as in Lemma \ref{f_zeros-outside}, we obtain

\begin{lemma}[Lemma 4.5 in \cite{M}]\label{zeros-outside}
For $s\in \{ s \mid \sigma<1-2m \} \cup \{ s \mid \sigma>2m \}$, the zeros of $Z_k(s)$ are all located in small circles centred negative even integers and positive odd integers,
and the number of those is $k$ in each circles.
\end{lemma}

Finally, we present a result of Matsumoto and Tanigawa on the zeros of $Z_k(s)$ in the ``critical strip''.
\begin{prop}[see the proof of Theorem 2 in \cite{M-T}]\label{zeros-Z_k}
If the Riemann Hypothesis is true, then the number of zeros of $Z_k(s)$ in $\{ s \mid 1-2m\leq \sigma \leq 2m, \sigma \neq1/2 \}$ is $O_k(1)$.
\end{prop}
They proved that the difference of the number of zeros of $Z^{(k)}(t)$ and that of $Z_k(s)$ is $O_k(1)$. From this fact, they obtained the following approximate formula;
\begin{prop}[Theorem 2 in \cite{M-T}]\label{number-zero}
Let $N_k(T)$ be the number of the zero of $Z^{(k)}(t)$ in the interval $(0,T)$. If the Riemann Hypothesis is true, then for any $k\geq 1$,
\begin{equation*}
N_k(T)=\frac{T}{2\pi}\log \frac{T}{2\pi}-\frac{T}{2\pi}+O_k(\log T).
\end{equation*}
\end{prop}

\section{An Alternative Proof of Matsuoka's result}
In this section, we give an alternative proof of Matsuoka's result;
\begin{theorem}[Theorem 1.1 in \cite{M}]\label{M-th}
If the Riemann Hypothesis is true, then for any non-negative integer $k$ there exists a $T=T(k)>0$ such that, for $t\geq T$, $Z^{(k+1)}(t)$ has exactly one zero between consecutive zeros of $Z^{(k)}(t)$.
\end{theorem}
More precisely, we prove Mozer's formula
\begin{equation}\label{Mozer}
\frac{d}{dt}\frac{Z^{(k+1)}}{Z^{(k)}}(t)=-\sum_{\gamma_k}\frac{1}{(t-\gamma_k)^2}+O_k(t^{-1})
\end{equation}
in the different way from Matsuoka's proof. Here the sum is taken over zeros of $Z^{(k)}(t)$.
If we can prove this formula, then we see that
\begin{align*}
\frac{d}{dt}\frac{Z^{(k+1)}}{Z^{(k)}}(t)
&<
-\sum_{0<\gamma_k<t}\frac{1}{(t-\gamma_k)^2}+At^{-1} \\
&<
t^{-1}(A-N_k(t)t^{-1}).
\end{align*}
By Proposition \ref{number-zero}, this is negative for large $t$.
Matsuoka's proof is inspired by Anderson's method (\cite{An}), which implies the above theorem for $k=1$. Matsuoka considered the integral
\begin{equation*}
\int_{C}\frac{G_k'}{G_k}(w)\frac{s}{w(s-w)}dw,
\end{equation*}
where $C$ is an appropriate rectangle and
\begin{equation*}
G_k(w)=h(w)\frac{Z_k(w)}{f_k(w)}
\end{equation*}
with $h(w)=\pi^{-s/2}\Gamma(s/2)$. However, in view of Theorem \ref{xi-facto}, we can prove the formula more easily.

\begin{proof}[Proof of (\ref{Mozer})]
By the definition of $\xi_k(s)$ and (\ref{Z^k}), we have
\begin{equation*}
\xi_k\left(\frac{1}{2}+it \right)=-i^{-k}\pi^{-\frac{1}{4}}\left(\frac{1}{4}+t^2 \right)\left|\Gamma \left(\frac{1}{4}+\frac{it}{2} \right) \right|^{-2k+1}Z^{(k)}(t).
\end{equation*}
Hence when we put
\begin{equation*}
g_k(t)=i^{-k}\pi^{-\frac{1}{4}}\left(\frac{1}{4}+t^2 \right)\left|\Gamma \left(\frac{1}{4}+\frac{it}{2} \right) \right|^{-2k+1}
\end{equation*}
then, by the logarithmic derivative with respect to $t$, we can obtain
\begin{equation}
i\frac{\xi_k'}{\xi_k}\left(\frac{1}{2}+it \right)=\frac{g_k'}{g_k}(t)+\frac{Z^{(k+1)}}{Z^{(k)}}(t).
\end{equation}
As for the function $g_k'/g_k(t)$, we can see that
\begin{equation*}
\frac{g_k'}{g_k}(t)
=
-(2k-1)\frac{d}{dt} \Re \log \Gamma \left(\frac{1}{4}+\frac{it}{2} \right)+\frac{2t}{\frac{1}{2}+t^2}
\end{equation*}
and hence
\begin{equation*}
\frac{d}{dt}\frac{g_k'}{g_k}(t)\ll_k t^{-1}
\end{equation*}
by Lemma \ref{estgam}.

On the other hand, Theorem \ref{xi-facto} implies that
\begin{equation*}
\frac{\xi_k'}{\xi_k}\left(\frac{1}{2}+it \right)=B_k+\sum_{\rho_k}\left(\frac{1}{\frac{1}{2}+it-\rho_k}+\frac{1}{\rho_k} \right).
\end{equation*}
Therefore we have
\begin{align*}
\frac{d}{dt}\frac{\xi_k'}{\xi_k}\left(\frac{1}{2}+it \right)
&=
\sum_{\rho_k}\frac{-i}{(\frac{1}{2}+it-\rho_k)^2} \\
&=
i\sum_{\gamma_k}\frac{1}{(t-\gamma_k)^2}+\sum_{\substack{\rho_k \\ \beta_k\neq \frac{1}{2}}}\frac{-i}{(\frac{1}{2}+it-\rho_k)^2} \\
&=
i\sum_{\gamma_k}\frac{1}{(t-\gamma_k)^2} \\
&\quad
+\sum_{\substack{\rho_k \\ \beta_k<1-2m, \ 2m<\beta_k}}\frac{-i}{(\frac{1}{2}+it-\rho_k)^2}+O(t^{-2}) \\
&=
i\sum_{\gamma_k}\frac{1}{(t-\gamma_k)^2}+O(t^{-1})
\end{align*}
by Proposition \ref{zeros-Z_k} and Lemma \ref{zeros-outside}. To show the last equality, we use the same argument as Matsuoka's (see p.15 in \cite{M});
\begin{equation*}
\sum_{\substack{\rho_k \\ \beta_k<1-2m, \ 2m<\beta_k}}\frac{1}{(\frac{1}{2}+it-\rho_k)^2}\ll_k \sum_{n=m}^{\infty}\frac{k}{(t+2n+1)^2}\ll_k\int_{2m+1}^{\infty}\frac{dx}{(t+x)^2}.
\end{equation*}
Thus we have
\begin{equation*}
\frac{d}{dt}\frac{Z^{(k+1)}}{Z^{(k)}}(t)=-\sum_{\gamma_k}\frac{1}{(t-\gamma_k)^2}+O_k(t^{-1}).
\end{equation*}
\end{proof}

\section*{Acknowledgements}
The author would like to thank my supervisor Professor Kohji Matsumoto for useful advice.
The author is grateful to the seminar members for some helpful remarks and discussions.

\end{document}